\documentclass[11pt,leqno]{amsart}
\usepackage{amsmath,amsfonts,amsthm,color,amscd,amssymb,graphicx,ulem}
\usepackage{epstopdf}
\numberwithin{equation}{section}
\usepackage{fullpage}

\newcommand\s{\sigma}

\renewcommand\a{\alpha}
\renewcommand\b{\beta}

\renewcommand\a{\alpha}

\renewcommand\b{\beta}

\newcommand\R{\mathbb R}

\newcommand\br{\begin{remark}}
\newcommand\er{\end{remark}}
\newcommand\bp{\begin{pmatrix}}
\newcommand\ep{\end{pmatrix}}
\newcommand{\be}{\begin{equation}}
\newcommand{\ee}{\end{equation}}
\newcommand\ba{\begin{equation}\begin{aligned}}
\newcommand\ea{\end{aligned}\end{equation}}

\newcommand{\bap}{\begin{app}}
\newcommand{\eap}{\end{app}}
\newcommand{\begs}{\begin{exams}}
\newcommand{\eegs}{\end{exams}}
\newcommand{\beg}{\begin{example}}
\newcommand{\eeg}{\end{exaplem}}
\newcommand{\bpr}{\begin{proposition}}
\newcommand{\epr}{\end{proposition}}
\newcommand{\bt}{\begin{theorem}}
\newcommand{\et}{\end{theorem}}
\newcommand{\bc}{\begin{corollary}}
\newcommand{\ec}{\end{corollary}}
\newcommand{\bl}{\begin{lemma}}
\newcommand{\el}{\end{lemma}}
\newcommand{\bd}{\begin{definition}}
\newcommand{\ed}{\end{definition}}
\newcommand{\brs}{\begin{remarks}}
\newcommand{\ers}{\end{remarks}}

\newcommand{\RR}{{\mathbb R}}

\newtheorem{theorem}{Theorem}[section]
\newtheorem{proposition}[theorem]{Proposition}
\newtheorem{corollary}[theorem]{Corollary}
\newtheorem{lemma}[theorem]{Lemma}

\theoremstyle{remark}
\newtheorem{remark}[theorem]{Remark}
\theoremstyle{definition}
\newtheorem{definition}[theorem]{Definition}

\newtheorem{example}[theorem]{Example}

\newcommand{\beq}{\begin{equation}}
\newcommand{\eeq}{\end{equation}}

\title{
Entropy criteria and stability of extreme shocks: 
a remark on a paper of Leger and Vasseur
}

\author{Benjamin Texier}
\address{ Universit\'e Paris-Diderot, Institut de Math\'ematiques de Jussieu, UMR CNRS 7586, and Ecole Normale Sup\'erieure, D\'epartement de Math\'ematiques et Applications, UMR CNRS 8553}
\email{ texier@math.jussieu.fr}
\thanks{Research of B.T. was partially supported by the Project ``Instabilities in Hydrodynamics'' funded by the Mairie de Paris (under the ``Emergences'' program) and the Fondation Sciences Math\'ematiques de Paris.}
\author{Kevin Zumbrun}
\address{Indiana University, Bloomington, IN 47405}
\email{kzumbrun@indiana.edu} 
\thanks{Research of K.Z. was partially supported
under NSF grants no. DMS-0300487 and DMS-0801745. K.Z. thanks the Ecole Normale Sup\'erieure, Ulm, and coordinating host Thomas Alazard
for their hospitality during the visit in which this
work was carried out.}

\thanks{Both authors thank the anonymous referee for helpful comments regarding the presentation. }
\begin{document}

\begin{abstract}
We show that 
a relative entropy condition recently shown by Leger and Vasseur
to imply uniqueness and stable $L^2$ dependence on initial data of
Lax $1$- or $n$-shock solutions of an $n\times n$ system of hyperbolic
conservation laws
with convex entropy
 implies Lopatinski
stability in the sense of Majda. This means in particular that Leger and Vasseur's relative entropy condition represents a considerable improvement over the standard entropy condition of decreasing shock strength and increasing entropy along forward Hugoniot curves, which, in a recent example exhibited by Barker, Freist\"uhler and Zumbrun, was shown to fail to imply Lopatinski stability, even for systems with convex entropy.
This observation bears also on the parallel question of existence,
at least for small $BV$ or $H^s$ perturbations. 
\end{abstract}
\date{\today}
\maketitle


\section{Introduction}
In this brief note, we examine for extreme Lax shock solutions of a
system of conservation laws
\be\label{system}
u_t + f(u)_x=0, \; u\in \RR^n,
\ee
possessing a convex entropy
\be\label{ent}
\eta, \quad P := \nabla^2 \eta>0,\quad \nabla_u \eta \nabla_u f= \nabla_u q,
\ee
the relation between Lopatinski stability in the sense of 
Erpenbeck and Majda \cite{Er,M1,M2,M3} and a relative entropy criterion
introduced recently by Leger and Vasseur \cite{LV}.

A number of entropy criteria have been proposed over the years to
distinguish physically admissible or stable shock waves.
Some of the oldest \cite{B,W}
are decrease of characteristic speed (compressivity) or increase in entropy
across the shock, or their instantaneous equivalents:
monotone decrease of shock speed or increase of entropy
along the forward $1$-Hugoniot curve from a fixed left state.
All of these conditions coincide for small-amplitude waves \cite{Sm},
agreeing also with the property of stability, or well-posedness
of the shock solution with respect to nearby initial data, as determined
by nonvanishing of a certain Lopatinski determinant \cite{La,Er,M1,M2,M3,Da};
however, for large-amplitude waves, the relations between these different
criteria are unclear.


{The related concept of entropy dissipation $\eta_t-q_x \leq 0$ 
in the sense of measures at the shock, or $[\eta]-\sigma[q]\leq 0$ for
shock speed $\sigma,$ is justified as a
necessary condition for the physicality criterion of {\it admissibility}. Here we are thinking of admissibility in the sense of Lax, for small-amplitude waves, under the assumption of genuine nonlinearity and strictly convex entropy (Theorem 5.7, \cite{La}), but also admissibility in the sense of existence of a nearby viscous shock profile, for arbitrary-amplitude waves, under the assumption of
an associated entropy-compatible viscosity \cite{K, KS}.}

The property of entropy dissipation 
may be seen to follow from monotone decrease
of shock speed along the Hugoniot curve (Lemma 3, \cite{LV}; see also \cite{La})
making a link (if only one way) between large-amplitude admissibility and 
the more classical monotonicity condition defined originally in a
small-amplitude context.

Most recently, Leger and Vasseur have introduced
a {\it relative entropy condition} pertaining to arbitrary-amplitude 
waves\footnote{See also an earlier analysis by Leger \cite{Leg}, proving $L^2$-contractivity of entropy solutions in the scalar case.}.
Specifically, assuming 
$L^\infty$ boundedness and
 a strong trace property on solutions\footnote{Satisfied by $BV$ solutions.}, these authors establish for systems
of conservation laws possessing a convex entropy uniqueness and
stable dependence 
in $L^2$ 
on initial data for perturbations 
of extreme Lax shock waves (without loss of generality $1$-shocks)
of arbitrary amplitude, under the conditions that 
\begin{itemize}
\item[(i)] shock speed is nonincreasing
along the forward Hugoniot curve from a fixed left state; explicitly $\s'(s) \leq 0$ for $s \geq 0,$ with notation introduced in Section \ref{s2} below, and 
\item[(ii)] relative entropy (defined in \eqref{etarel}) is nondecreasing with respect to the right state along
the forward Hugoniot curve; explicitly $d_s \eta(u|S_u(s)) \geq 0$ for $s \geq 0.$
\end{itemize}
The result of Leger and Vasseur is nominally an a priori short-time stability estimate and leaves open the question of existence, even for more standard 
perturbations that are small in $BV$ or $H^s.$

The purpose of the present note is to verify that conditions (i)-(ii) imply satisfaction of the Lopatinski condition of Majda \cite{M1,M2,M3,Le1}.
Indeed, we show that satisfaction of the Lopatinski condition follows
from the much weaker conditions that (i) and (ii) hold only for the single value $s = s_+$ corresponding to the right endstate; see (i')-(ii') below. 


This observation immediately yields, by the existing theory of
\cite{Le1,M1,M2,M3}, existence and stability
for the classes of small $BV$ or $H^s$ perturbations under conditions
(i')-(ii').
One might hope that it could be eventually of use also in constructing
approximate solutions and ultimately the demonstration of existence
in the much more
delicate $L^2$/strong trace setting considered in \cite{LV}.

We remark that Lopatinski stability is necessary for the physicality 
condition of {\it viscous stability}, or stability of an associated
viscous profile \cite{ZS}.

\br
From the above discussion, (i) is sufficient for entropy dissipation,
which is necessary for the physicality
condition of admissibility, and (i')-(ii') are sufficient for 
Lopatinski stability, which is necessary for the 
physicality condition of viscous stability.
As strict entropy dissipation and Lopatinski stability are open conditions,
whereas (i) and (i')-(ii'), as nonstrict monotonicity conditions,
are closed, it is evident that {\it {\rm (i)} and {\rm (i')-(ii'),} 
are sufficient but not necessary} 
for strict entropy dissipation and Lopatinski stability, respectively.
For, a closed condition holds at the boundary of its region of satisfaction,
whereupon an implied open condition must therefore hold at some point
outside.
\er

\br \label{r31}
An examination of the argument of \cite{LV}, reveals that their hypothesis (ii) may be weakened to\footnote{Specifically, in the course of the proof of Theorem 3 in \cite{LV}, assumption (ii) is used only in the key Lemma 4, which is invoked only in Lemma 8, in 1-3 page 291, where the weakened form (ii$^*$) is used, not for all $s \geq 0$ but only in an interval $0 \leq s \leq s_+ + C$ (in their notation, $0 \leq s \leq s_u$), where $C > 0$ is determined by the $L^\infty$ bound assumed on solutions.} 
$$\mbox{\rm (ii$^*$)} \quad  (s - s_+) \big( \eta(u|S_{u}(s))-\eta(u|S_u(s_+)) \big) \geq 0,  \quad  \mbox{for $s \geq 0,$}$$
with notation introduced in Section \ref{s2}, where $(u, S_u(s_+), \s(s_+))$ is the fixed 1-shock under consideration.
Evidently, (ii$^*$) implies our condition (ii') stated in Assumption II below. 
\er

\section{Definitions and result} \label{s2} 

Let $\eta,q$ be a convex entropy/entropy flux pair.
Then \cite{La}, for $P = \nabla^2 \eta$, $A:=\nabla f$,
$P$ is symmetric positive definite
and $PA$ is symmetric. 
Thus, $A$ is self-adjoint with respect to the inner product induced by $P$,
and so the eigenvectors of $A$ corresponding to distinct eigenvalues are $P$-orthogonal.
The {\it relative entropy} $\eta(u|v)$ is defined 
following \cite{D,LV} as
\be\label{etarel}
\eta(u|v):=\eta(u)-\eta(v)-\nabla \eta(v)(u-v).
\ee

We assume that $A$ is strictly hyperbolic\footnote{For Theorem \ref{main} to hold, we only need strict hyperbolicity to hold at the right endstate $S_u(s_+)$ introduced in Assumption II.}, with eigenvalues $a_1<a_2<\dots<a_n,$ and associated eigenvectors $r_1,r_2,\dots,r_n.$ 

\medskip
{\bf Assumptions I.}
For a given left state $u$, suppose that there is a well-defined
$C^1$
$1$-Hugoniot curve of states $S_u(s)$ and associated speeds
$\sigma(s)$,  $s\geq 0$, satisfying
\be\label{rh}
\sigma(s)(S_u(s)-u)=f(S_u(s))-f(u),
\ee
with $S_u(0)=u$, $\sigma(0)=a_1(u)$.
More precisely, assume that condition \eqref{rh} is everywhere full rank,
with the linearized equations
\be\label{lrh}
\sigma'(s) (S_u(s)-u) = \big(A(S_u(s))-\sigma(s) \big)S_u'(s)
\ee
(hence, by the Implicit Function Theorem, also 
the nonlinear equations \eqref{rh})
uniquely solvable (up to a constant multiplier) for $(\sigma'(s), S_u'(s))$.
Moreover, assume that the resulting discontinuity is a Lax $1$-shock, in
the sense that 
\be\label{lax}
 a_1(u)>\sigma(s) \quad \hbox{\rm  and} \quad a_1(S_u(s)) <\sigma(s)<
a_2(S_u(s)) < a_3(S_u(s)) < \dots < a_n(S_u(s)).
\ee

\medskip

The {\it Lopatinski} (stability) {\it condition} for the
shock $(u,S_u(s),\sigma(s)),$ with notation introduced in Assumptions I above, is 
\be\label{lop}
\det \bp (S_u(s)-u) & r_2(S_u(s)) & \dots & r_n(S_u(s))\ep \neq 0.
\ee

This may be recognized as the condition that the Riemann problem
be well-posed for data near $(u,S_u(s))$, more
precisely, that the Jacobian of the associated Lax wave-map 
be full rank \cite{La,Sm}.

Condition \eqref{lop} is a crucial 
building block both (through resulting a priori stability estimates)
for the small $H^s$-perturbation existence/stability 
theory of Majda \cite{M1,M2,M3,Me} and
(through direct construction based on Riemann solutions) in 
the small-$BV$ perturbation existence/stability 
theory of Lewicka and others \cite{Le1,Le2,C}
in the vicinity of a single large-amplitude shock.

\medskip
{\bf Assumptions II.}
(i') $\sigma'(s_+)\leq 0$,
(ii') $d_s \eta(u | S_u(s_+))\geq 0$, for shock $(u,S_u(s_+),\s(s_+))$, 
$s_+ \geq 0$.

\begin{lemma}\label{l}
Condition {\rm (ii')} is equivalent 
to 
\be\label{form}
\langle S_u'(s_+), \nabla^2\eta(S_u(s_+)) (S_u(s_+)-u) \rangle \geq 0.
\ee
\end{lemma}

\begin{proof}
Differentiating \eqref{etarel}, we have
$$
\begin{aligned}
d_s \eta(u|S_u(s))
 &= d_s \big[ \eta(u)-\eta(S_u(s))-\nabla \eta(S_u(s)) (u-S_u(s)) \big] \\
&=
-\nabla \eta(S_u(s))S_u'(s)
- \big[ \big\langle S_u'(s),\nabla^2 \eta (S_u(s))(u-S_u(s))\big\rangle
-\nabla \eta(S_u(s))S_u'(s) \big] \\
&=
-  \big\langle S_u'(s),\nabla^2 \eta(S_u(s))(u-S_u(s)) \big\rangle,
\end{aligned}
$$
whence the assertion follows.
\end{proof}

\begin{theorem}\label{main}
Under Assumptions {\rm I} and {\rm II,} Lopatinski condition \eqref{lop} holds 
for $(u,S_u(s_+),\sigma(s_+))$.
\end{theorem}

\begin{proof} Failure of \eqref{lop} implies that 
 $$ S_u(s_+) - u = \sum_{2 \leq j \leq n} \a_j r_j^+, \qquad  r_j^+ := r_j(S_u(s_+)),$$
 for some $\a_j \in \R,$ which are not all equal to zero, since we may assume $s_+ > 0,$ $S_u(s_+) \neq u.$ 
 Then, by Assumptions I, 
 $$ (A_+ - \s(s_+)) S'_u(s_+) = \s'(s_+) \sum_{2 \leq j \leq n} \a_j r_j^+, \qquad A_+ := A(S_u(s_+)).$$
 Hence, inverting $A_+ - \s(s_+)$ (as we may by \eqref{lax}): 
$$ 
  S'_u(s_+) = \s'(s_+) \sum_{2 \leq j \leq n} \b_j r_j^+, \qquad \b_j := (a_j(S_u(s_+))- \s(s_+))^{-1} \a_j.
$$ 
We remark that, since $(u, S_u(s_+), \s(s_+))$ is a $1$-shock \eqref{lax}, there holds $\b_j \a_j \geq 0,$  and, since the shock is non-trivial, $\b_j \a_j > 0$ for at least one $j.$ Hence, by $\s'(s_+) \neq 0$ (a consequence of Assumptions I), positive definiteness of $P,$ and the fact that the eigenvectors of $A$ are $P$-orthogonal, we deduce 
\be \label{contr} 
 \big\langle S'_u(s_+), P_+ (A_+ - \s(s_+)) S_u'(s_+) \big\rangle  = \s'(s_+)^2 \sum_{2 \leq j \leq n} \b_j \a_j \big\langle r_j^+, P_+ r_j^+\big\rangle > 0,  
\ee
with notation $P_+ := P(S_u(s_+)) = \nabla^2 \eta(S_u(s_+)).$
However, Assumptions II and Lemma \ref{l} imply
$$ \s'(s_+) \big\langle S'_u(s_+), P_+ (S_u(s_+) - u) \big\rangle\leq 0,$$
so that, by \eqref{lrh},
$$ \big\langle S'_u(s_+), P_+( A_+ - \s(s_+)) S'_u(s_+) \big\rangle \leq 0,$$
in contradiction with \eqref{contr}.
\end{proof}

\section{discussion and open problems}
The corresponding absolute entropy conditions 
that shock strengh is decreasing
and 
 absolute entropy $\eta(S_u(s))$ is increasing along the forward
Hugoniot curve have been shown \cite{B} to hold {\it globally}
for very general gas dynamical equations of state.
However, recently, Barker, Freist\"uhler, and Zumbrun 
\cite{BFZ}
have shown by explicit example that there exist systems satisfying these
conditions and also possessing a convex entropy, but for which nonetheless
{\it the Lopatinski condition can fail}.
Thus, the relative entropy condition represents a considerable sharpening
of the older absolute entropy condition.

An interesting open problem would be to find an analog of this
condition for intermediate shocks; however, we see no obvious candidate for
this.  Certainly, the approach of Theorem \ref{main} breaks down,
since there is no relation between $P(u)$- and $P(S_u(s))$-orthogonality.
An interesting, but more speculative, problem
would be to make use of the Lopatinski condition to construct 
approximate solutions in the small $L^2$-perturbation class, 
towards an eventual possible small $L^2$/strong trace class 
existence theory.
It would be extremely interesting, of course, to find some analog
also for the corresponding viscous shock stability problem, whether
directly as in \cite{LV}, or indirectly as here through the study of spectral
stability and the linearized eigenvalue problem.

\end{document}